\documentclass[a4paper,11pt]{article}

\usepackage[mac]{inputenc}
 \usepackage[T1]{fontenc}
 \usepackage[normalem]{ulem}
 \usepackage[english]{babel}
\usepackage{amsmath}
  \usepackage{amsthm}
 \usepackage{bbm}
 \usepackage{amssymb}
 \usepackage{verbatim}
 \usepackage{vmargin}
 \usepackage{graphicx}
 \usepackage{color}
 \usepackage{epsfig}
 \usepackage[only,llbracket,rrbracket]{stmaryrd}
 \usepackage{enumerate}
\usepackage{tikz}
\usepackage{setspace}

\usepackage{setspace}

\newtheorem{theorem}{Theorem}
\newtheorem*{definition}{Definition}

\newtheorem{proposition}{Proposition}
\newtheorem{lemma}{Lemma}
\newtheorem{corollary}{Corollary}
\numberwithin{equation}{section}

\theoremstyle{theorem}
\newtheorem{remarks}{\textbf{Remarks}}[section]{\vskip 0.5cm}

\newtheorem{remark}{\textbf{Remark}}{\vskip 0.5cm} 
\newtheorem*{acknowledgements}{\textbf{Acknowledgements}}{\vskip 0.5cm}

\title{$L^1$ averaging lemma for transport equations with Lipschitz force fields}
\author{Daniel Han-Kwan\footnote{\'Ecole Normale Sup\'erieure, D\'epartement de Math\'ematiques et Applications,  45 rue d'Ulm 75230 Paris Cedex 05 France, email : {daniel.han-kwan@ens.fr}}}
 \date{}
\begin{document}
 \maketitle

\begin{abstract}
The purpose of this note is to extend the $L^1$ averaging lemma of Golse and Saint-Raymond \cite{GolSR} to the case of a kinetic transport equation with a force field $F(x)\in W^{1,\infty}$. To this end, we will prove a local in time mixing property  for the transport equation $\partial_t f + v.\nabla_x f + F.\nabla_v f =0$.
\end{abstract}

\section*{Introduction}
Let $d \in \mathbb{N}^*$ and $1<p<+\infty$. We consider $\mathbb{R}^d$ equipped with the Lebesgue measure.  Let $f(x,v)$ and $g(x,v)$ be  two measurable functions in $L^p(\mathbb{R}^d \times \mathbb{R}^d)$  satisfying the transport equation:
\begin{equation}
v.\nabla_x f = g.
\end{equation}

 Although transport equations are of hyperbolic nature (and thus there is a priori no regularizing effect), it was first observed for  by Golse, Perthame and Sentis in \cite{GPS} and then by Golse, Lions, Perthame and Sentis \cite{GLPS} (see also Agoshkov \cite{Ago} for related results obtained independently) that the velocity average (or moment) $\rho(x)= \int f \Psi(v) dv$ with $\Psi \in \mathcal{C}^\infty_c(\mathbb{R}^d)$ is smoother than $f$ and $g$ : more specifically it belongs to some Sobolev space $W^{s,p}(\mathbb{R}^d)$ with $s>0$.
These kinds of results are referred to as "velocity averaging lemma". The analogous results in the time-dependent setting also hold, that is for the equation:
\begin{equation}
\label{time}
\partial_t f + v.\nabla_x f = g.
\end{equation}

Refined results with various generalizations (like derivatives in the right-hand side, functions with different integrability in $x$ and $v$...) were obtained in \cite{DPLM}, \cite{BZ}, \cite{PS}, \cite{JV}. There exist many other interesting contributions. We refer to Jabin \cite{J} which is a rather complete review on the topic.

Velocity averaging lemmas are tools of tremendous importance in kinetic theory since they provide some strong compactness  which is very often necessary to study non-linear terms (for instance when one considers an approximation scheme to build weak solutions, or for the study of asymptotic regimes). There are numerous applications of these lemmas; two emblematic results are the existence of renormalized solutions to the Boltzmann equation \cite{DPLbol} and the existence of global weak solutions to the Vlasov-Maxwell system \cite{DPLvm}. Both are due to DiPerna and Lions.

The limit case $p=1$ is actually of great interest. In general, for a sequence $(f_n)$ uniformly bounded in $L^1(dx\otimes dv)$ with $v.\nabla_x f_n$ also uniformly bounded in $L^1 (dx\otimes dv) $, the sequence of velocity averages $\rho_n=\int f_n \Psi(v)dv$  is not relatively compact in $L^1(dx)$ (we refer to \cite{GLPS} for an explicit counter-example). This lack of compactness is due to the weak compactness pathologies of $L^1$. Indeed, as soon as we add some weak compactness to the sequence (or equivalently some equiintegrability in $x$ and $v$ in view of the classical Dunford-Pettis theorem), then we recover some strong compactness in $L^1$ for the moments (see Proposition 3 of \cite{GLPS} or Proposition \ref{equiXV} below).  
  
  We recall precisely the notion of equiintegrability which is central in this paper.

  \begin{definition}\label{equi} 
  
  \begin{enumerate}
 \item (Local equiintegrability in $x$ and $v$)
 
 Let  $(f_\epsilon)$ be a bounded family of
 $L^1_{loc}(dx\otimes dv)$. It is said locally equiintegrable in $x$ and
$v$ if and only if for any $\eta>0$ and for any compact subset
$K\subset \mathbb{R}^{d}\times  \mathbb{R}^{d}$, there exists $\alpha>0$  
such that for any measurable set  $A\subset  \mathbb{R}^{d}\times  \mathbb{R}^{d}$ with  $\vert A\vert < \alpha$, 
we have for any $\epsilon$ :
\begin{equation}
\int_A \mathbbm{1}_K(x,v)\vert f_\epsilon(x,v)\vert dv  
dx \leq \eta.
\end{equation}

  \item (Local equiintegrability in $v$)
  
Let  $(f_\epsilon)$ be a bounded family of
 $L^1_{loc}(dx\otimes dv)$. It is said locally equiintegrable in
$v$ if and only if for any $\eta>0$ and for any compact subset
$K\subset \mathbb{R}^{d}\times  \mathbb{R}^{d}$, there exists $\alpha>0$  
such that for each family $(A_x)_{x \in \mathbb{R}^d}$ of measurable sets of  $\mathbb{R}^{d}$  satisfying  $\sup_{x \in \mathbb{R}^d} \vert A_x\vert < \alpha$, 
we have for any $\epsilon$ :
\begin{equation}
\int \left(\int_{A_x} \mathbbm{1}_K(x,v)\vert f_\epsilon(x,v)\vert dv  
\right)dx \leq \eta.
\end{equation}

\end{enumerate}

\end{definition}

We observe that local equiintegrability in $(x,v)$ always implies local equiintegrability in $v$, whereas the converse is false in general.

The major improvement of the paper of Golse and Saint-Raymond \cite{GolSR} is to show that actually, only equiintegrability in $v$ is needed to obtain the $L^1$ compactness for the moments. This observation was one of the key arguments of their outstanding paper \cite{GSR} which establishes the convergence of renormalized solutions to the Boltzmann equation in the sense of DiPerna-Lions to weak solutions to the Navier-Stokes equation in the sense of Leray.

More precisely, the result they prove is Theorem \ref{L1} stated afterwards, with $F=0$ (free transport case). The aim of this paper is to show that the result also holds if one adds some force field $F(x)=(F_i(x))_{1\leq i \leq d} $ with $F \in W^{1,\infty}(\mathbb{R}^d)$: 

  \begin{theorem}\label{L1} Let  $(f_\epsilon)$ be a family bounded 
in  $L^1_{{loc}}(dx\otimes dv)$  locally
equiintegrable in $v$ and such that $v.\nabla_x f_\epsilon+ F. \nabla_v f_\epsilon$ is bounded in $L^1_{{loc}}(dx\otimes
dv)$. Then :
\begin{enumerate}
   \item $(f_\epsilon)$ is locally equiintegrable in both variables $x$
\textbf{and} $v$.
   \item For all $\Psi \in \mathcal{C}^{1}_c(\mathbb{R}^d)$, the family
$\rho_\epsilon(x)=\int f_\epsilon(x,v)\Psi(v)dv$ is relatively compact in $L^1_{{loc}}(dv)$.
\end{enumerate}
\end{theorem}

One key ingredient of the proof for $F=0$ is the nice dispersion properties of the free transport operator. We will show in Section \ref{mix} that an analogue also holds for small times when $F\neq 0$:

\begin{proposition}
 Let $ F(x) $ be a Lipschitz vector field. There exists a maximal time $\tau>0$  (depending only on $\Vert \nabla_x F \Vert _{L^{\infty}}$)  such that, if $f$ is the solution to
the transport equation:

\[
 \left\{
    \begin{array}{ll}
      \partial_t f + v.\nabla_x f +F.\nabla_v f=0, \\
      f(0,.,.)=f^0\in L^p(dx \otimes dv),
    \end{array}
  \right.
\]
Then:

\begin{equation}
\label{intromix}
\forall \vert t \vert \leq \tau, \Vert{f(t)}\Vert_{L^\infty_x(L^1_v)} \leq
\frac{2}{|t|^{d}} \Vert{f^0}\Vert_{L^1_x(L^\infty_v)}.
\end{equation}
\end{proposition}

Let us also mention that the main theorem generalizes to the time-dependent setting, for transport equations of the form (\ref{time}). The usual trick to deduce such a result from the stationary case is to enlarge the phase space. Indeed we can consider $x'= (t,x)$ in $\mathbb{R}^{d+1}$ endowed with the Lebesgue measure, and $v'=(t,v)$ in $\mathbb{R}^{d+1}$ endowed with the measure $\mu=\delta_{t=1} \otimes \text{Leb}$ (where $\delta$ is the dirac measure). Then such a measure $\mu$ satisfies property (2.1) of \cite{GLPS}. As a consequence, all the results of Section \ref{first} will still hold.

Nevertheless, we observe that our key local in time mixing estimate (\ref{intromix}) seems to not hold when $\mathbb{R}^{d+1}$ is equipped with the new measure $\mu$ (the main problem being that the only speed associated to the first component of $v'$ is $1$). For this reason, we can not prove that equiintegrability in $v$ implies equiintegrability in $t$.  One result (among other possible variants) is the following:

  \begin{theorem} 
  \label{L1time}
  Let $F (t,x)\in \mathcal{C}^0(\mathbb{R}^+, W^{1,\infty}(\mathbb{R}^d))$. Let  $(f_\epsilon)$ be a family bounded 
in  $L^1_{{loc}}(dt\otimes dx\otimes dv)$  locally
equiintegrable in $v$ and such that $\partial_t f_\epsilon + v.\nabla_x f_\epsilon+ F. \nabla_v f_\epsilon$ is bounded in $L^1_{{loc}}(dt\otimes dx\otimes
dv)$. Then :
\begin{enumerate}
   \item $(f_\epsilon)$ is locally equiintegrable in the variables $x$ and $v$ (but not necessarily with respect to $t$).
   \item For all $\Psi \in \mathcal{C}^{1}_c(\mathbb{R}^d)$, the family
$\rho_\epsilon(t,x)=\int f_\epsilon(t,x,v)\Psi(v)dv$ is relatively compact  with respect to the $x$ variable in $L^1_{{loc}}(dt\otimes dx)$, that is, for any compact $K \subset \mathbb{R}^+_t \times \mathbb{R}^d_x$:

\begin{equation}
\lim_{\delta \rightarrow 0} \sup_{\epsilon} \sup_{\vert x' \vert \leq \delta} \Vert (\mathbbm{1}_{ K}\rho_\epsilon)(t,x+x')- (\mathbbm{1}_{ K}\rho_\epsilon)(t,x)\Vert_{L^1(dt\otimes dx)} =0.
\end{equation}

\end{enumerate}

\end{theorem}

Another possibility is to assume that  $f_\epsilon$ is bounded in $L^\infty_{{t,loc}}(L^1_{x,v,loc})$, in which case we will get equiintegrability in $t,x$ and $v$ and thus compactness for $\rho_\epsilon$ in $t$ and $x$. We refer to \cite{GSR}, Lemma 3.6, for such a statement in the free transport case.

\begin{remarks}
\label{rk}
\begin{enumerate}
\item Since the result of Theorem \ref{L1} is essentially of local nature, we could slightly weaken the assumption on $F$:

For any $R>0$,
\begin{equation}
 \exists M(R),\forall \vert x_1\vert ,\vert x_2\vert \leq R, \quad \vert F(x_1) - F(x_2) \vert  \leq M(R)   \vert x_1-x_2 \vert.
\end{equation}

%\begin{equation}
%\text{(H1)  :   } \exists M_2 \text{   s.t.   }  \frac{\vert F(x)\vert}{1+\vert x \vert} \leq M_2  \text{  ,  } \forall x\in \mathbb{R}^d
%\end{equation}

In other words we can deal with $ F \in W^{1,\infty}_{loc}$.

\item With the same proof, we can treat the case of force fields $F(x,v) \in W^{1,\infty}_{x,v,loc}$ with zero divergence in $v$ : 

\[
\operatorname{div}_v F=0.
\]
 
Typically we may think of the Lorentz force $v\wedge B$ where $B$ is a smooth magnetic field.

\item We can handle a family of force fields $(F_\epsilon)$ depending on $\epsilon$ as soon as $(F_\epsilon)$ is uniformly bounded in $ W^{1,\infty}(\mathbb{R}^d)$.

\end{enumerate}
\end{remarks}

The following of the paper is devoted to the proof of Theorem \ref{L1}. In Section \ref{first}, we prove that a family satisfying the assumptions of Theorem \ref{L1} and in addition locally equiintegrable in $x$ and $v$, has moments which are relatively strongly compact in $L^1$. In Section \ref{mix}, we investigate the local in time mixing properties of the transport equation $\partial_t f + v.\nabla_x f + F.\nabla_v f =0$. Finally in the last section, thanks to the mixing properties we establish, we show by an interpolation argument that equiintegrability in $v$ provides some equiintegrability in $x$.

\section{A first step towards $L^1$ compactness}
\label{first}

The first step is to show that under the assumptions of Theorem \ref{L1}, point 1 implies point 2. Using classical averaging lemma in $L^2$ (\cite{DPLvm}, \cite{DPLM}), we first prove the following $L^2$ averaging lemma.

%\begin{notation}
%We say that a measurable function $\psi$ belongs to $L^\infty_K$ if $\psi$ belongs to $L^\infty_K$ and if $\psi$ is compactly supported.
%\end{notation}

\begin{lemma} 
\label{lemL2}
Let $f,g \in L^2(dx \otimes dv)$ satisfy the transport equation:
\begin{equation}
  v.\nabla_x f + F. \nabla_v f = g.
\end{equation}

Then for all $\Psi \in \mathcal{C}^{1}_c(\mathbb{R}^d)$, $\rho(x)=\int
f(x,v)\Psi(v)dv \in H^{1/4}_x$. Moreover,
\begin{equation}
\Vert \rho \Vert_{H^{1/4}_x} \leq C\left(
\left\Vert F \right\Vert_{L^{\infty}_x} 
\Vert f \Vert_{L^2_{x,v}} +
\Vert g \Vert_{L^2_{x,v}}\right).\end{equation} ($C$ is a constant depending only on $\Psi$.)
\end{lemma}

\begin{proof}
The standard idea is to consider $-F.
\nabla_v f + g$ as a source. Then, since $\operatorname{div}_v F(x)=0$, we have :
\begin{eqnarray*}
-F. \nabla_v f + g&=& -\sum_{i=1}^d
\frac{\partial}{\partial_{v_i}}\left(F_i f 
\right)+g.
\end{eqnarray*}
We conclude by applying the $L^2$ averaging lemma of \cite{DPLvm}, Theorem 3.
\end{proof}

We recall now in Proposition \ref{lemV} an elementary and classical representation result, obtained by the method of characteristics. 

Let
$b=(v,F)$, $Z=(X,V)$. Since $F \in W^{1,\infty}$, $b$ satisfies the hypotheses of the global Cauchy-Lipschitz theorem. We therefore consider the trajectories 
defined by:

\begin{equation}
\label{carac} \left\{
    \begin{array}{ll}
      Z'(t;x_0,v_0)=b(Z(t;x_0,v_0)) \\
      Z(0;x_0,v_0)=(x_0,v_0).
    \end{array}
  \right.
\end{equation}

For all time, the application $(x_0,v_0) \mapsto
Z(t;x_0,v_0)=(X(t;x_0,v_0),V(t;x_0,v_0))$ is well-defined and is a $C^1$ diffeomorphism. Moreover, since $b$ does not depend explicitly on time, it is also classical that $Z(t)$ is a group. The inverse is thus given by $(x,v) \mapsto Z(-t;x,v)$.

\begin{remark}Since
$\mbox{div}(b)=0$,  Liouville's theorem shows that the volumes in the phase space are preserved (the jacobian determinant of  $Z$ is equal to $1$).
\end{remark}

\begin{proposition}
\label{lemV}
\begin{enumerate}
   \item The time-dependent Cauchy problem :
    \begin{equation}
\label{transport1} \left\{
    \begin{array}{ll}
      \partial_t f + v.\nabla_x f +F.\nabla_v f=0, \\
      f(0,.,.)=f^0\in L^p(dx \otimes dv)
    \end{array}
  \right.
\end{equation}
has a unique solution (in the distributional sense) represented by

$$f(t,x,v)=f^0(X(-t;x,v)),V(-t;x,v))\in L^p(dx\otimes dv).$$

\item For any $ \lambda > 0$, the transport equation

\begin{equation}
\label{transport2} \lambda f(x,v) + v.\nabla_x f +F.\nabla_v f=g \in L^p(dx\otimes dv)
\end{equation}
has a unique solution (in the distributional sense) represented by:
$$R_\lambda :g(x,v) \mapsto f(x,v)=\int_0^{+\infty}e^{-\lambda  
s}g(X(-s;x,v),V(-s;x,v))ds
\in L^p(dx\otimes dv).$$
In addition, $R_\lambda$ is a linear continuous map on $L^p$
with a norm equal to $\frac{1}{\lambda}$.
\end{enumerate}
\end{proposition}

Using Rellich's compactness theorem, we straightforwardly have the following corollary:
\begin{corollary}
\label{corL2}
The linear continuous map $
T_{\lambda,\Psi}$ :
$$L^2(dx \otimes dv) \rightarrow L^2_{{loc}}(dx)$$
$$ g \mapsto \rho=\int{R_\lambda(g)(.,v) 
\Psi(v)dv}$$ 
is compact for all $\Psi \in
\mathcal{C}^1_{c}(\mathbb{R}^d)$ and all $\lambda >0$.
\end{corollary}

\begin{proof}
Using Lemma \ref{lemL2} and Proposition \ref{lemV}, we have:
$$
\Vert T_{\lambda,\Psi}(g) \Vert_{H^{1/4}_x} \leq C \left( 1 + \Vert F \Vert_{L^\infty} \right) \Vert g \Vert_{L^2_{x,v}}.
$$
The conclusion follows.
\end{proof}

Using this compactness property, as in Proposition 3 of \cite{GLPS}, we can show the next result:

\begin{proposition}
\label{equiXV}
 Let $\mathcal{K}$ be a bounded subset of  $L^1(dx \otimes dv)$ equiintegrable in $x$ and $v$ (in view of the Dunford-Pettis theorem, it means in other words that $\mathcal{K}$ is weakly compact in $L^1$),
then $T_{\lambda,\Psi}(\mathcal{K})$ is relatively strongly compact in $L^1_{loc}(dx)$.
\end{proposition}

\begin{proof} We recall the proof of this result for the sake of completeness.

The proof is based on a real interpolation argument. We fix a parameter $\eta>0$. For any $g \in
\mathcal{K}$ and any $\alpha
 >0$, we may write :
\begin{equation*}
g=g_1^\alpha +g_2^\alpha,
\end{equation*}
with
\begin{eqnarray*}
g_1^\alpha&=&\mathbbm{1}_{\{|g(x,v)|>\alpha\}}g, \\
g_2^\alpha&=&\mathbbm{1}_{\{|g(x,v)|\leq\alpha\}}g.
\end{eqnarray*}

Then, by linearity of $T_{\lambda,\Psi}$, we write $u=T_{\lambda,\Psi} (g)=u_1+ u_2$, with
$u_1=T_{\lambda,\Psi}(g_1^\alpha)$ and $u_2=T_{\lambda,\Psi}(g_2^\alpha)$.

Let $K$ be a fixed compact set of $\mathbb{R}^{d}_x$.

We clearly have, since $T_{\lambda,\Psi}$ is linear continuous on $L^1(K)$:
\begin{equation*}
\Vert u_1 \Vert_{L^1_x(K)} \leq C\Vert g_1^\alpha \Vert_1.
\end{equation*}
We notice that:
$$\left\vert \{(x,v), |g(x,v)|>\alpha\}\right\vert \leq \frac{1}{\alpha} \Vert g
\Vert_{L^1} \leq \frac{1}{\alpha} C.$$

Since $\mathcal{K}$ is equiintegrable, there exists $
\alpha
 >0$ such that for any $ g \in \mathcal{K}$ :
\begin{equation*}
\int |g\mathbbm{1}_{\{|g(x,v)|>\alpha\}}|dxdv \leq \frac{\eta}{C}.
\end{equation*}

Consequently for $\alpha$ large enough, we have:
\begin{equation*}
\Vert u_1 \Vert_{L^1_x(K)} \leq \eta.
\end{equation*}

The parameter $\alpha$ being fixed, we clearly see that $\{g_2^\alpha, g\in \mathcal{K}\}$
is a bounded subset of  $L^1_{x,v}\cap L^\infty_{x,v}$, and consequently of $L^2_{x,v}$. Because of Corollary \ref{corL2}, $\{u_2, u_2=T_{\lambda,\Psi}( g_2^\alpha), g
\in
\mathcal{K}\}$ is relatively compact in $L^2_{{loc}}(dx)$. In particular it is relatively compact in $L^1_{loc}(dx)$.

As a result, we have shown that for any $\eta >0$, there exists $\mathcal{K}_\eta \subset
L^1_x(K)$ compact, such that $T_{\lambda,\Psi}(\mathcal{K}) \subset \mathcal{K}_\eta +
B(0,\eta)$. So this family is precompact and consequently it is compact since $L^1_x(K)$ is a Banach space.

\end{proof}
We deduce the preliminary result (which means that the first point implies the second in Theorem \ref{L1}):

    \begin{theorem}
    \label{L1faible}
Let $(f_\epsilon)$ a family of $L^1_{loc}(dx\otimes dv)$ locally
equiintegrable in $x$ and $v$ such that $(v.\nabla_x f_\epsilon +
F.\nabla_v f_\epsilon)$ is a bounded family of
$L^1_{loc}(dx\otimes dv)$. Then for all $\Psi \in \mathcal{C}^{1}_c(\mathbb{R}^d)$,
the family $\rho_\epsilon(x)=\int f_\epsilon(x,v)\Psi(v)dv$ is relatively compact in 
$L^1_{{loc}}(dx)$.
\end{theorem}

\begin{proof}
Let $\Psi \in\mathcal{C}^{1}_c(\mathbb{R}^d)$. Let $R>0$ be a large number such that $\operatorname{Supp} \Psi \subset B(0,R)$; we intend to show that
$(\mathbbm{1}_{B(0,R)}(x)\rho_\epsilon(x))$ is compact in $L^1(B(0,R))$. First of all, 
we can assume that the $f_\epsilon$ are compactly supported 
in the same compact set $K \subset
\mathbb{R}^d\times \mathbb{R}^d$, with $B(0,R)\times B(0,R) \subset \mathring{K}$. Indeed we can multiply
the family by a smooth function $\chi$ such that:
\begin{eqnarray*}
 \operatorname{supp} \chi \subset K, \\ 
  \chi \equiv 1 \text{   on   } B(0,R)\times B(0,R).
\end{eqnarray*}

We observe that :
\begin{eqnarray*}
v.\nabla_x(\chi f_\epsilon)&=&\chi(v.\nabla_x
f_\epsilon)+f_\epsilon(v.\nabla_x \chi), \\
F.\nabla_v(\chi f_\epsilon)&=& \chi F.\nabla_v(
f_\epsilon)+f_\epsilon (F. \nabla_v \chi).
\end{eqnarray*}
Thus the family $(\chi f_\epsilon)$ satisfies the same $L^1$ boundedness properties as $(f_\epsilon)$. The equiintegrability property
is also clearly preserved. Furthermore, for any $x$ in $B(0,R)$, we have :
$$\int f_\epsilon(x,v) \Psi(v) dv = \int f_\epsilon(x,v)\chi(v)
\Psi(v) dv$$ Consequently we are now in the case of functions supported in the same compact set.

\par We have for all $
\epsilon>0, \lambda>0$, by linearity of the resolvent $R_\lambda$ defined in Proposition \ref{lemV} :
\begin{equation*}
\begin{split}
\int f_\epsilon(x,v) \Psi(v) dv =& \int R_\lambda (\lambda
f_\epsilon
+ v.\nabla_x f_\epsilon + F.\nabla_v f_\epsilon)\Psi(v)dv \\
=& \lambda \int (R_\lambda f_\epsilon)(x,v)\Psi(v)dv + \int
(R_\lambda (v.\nabla_x f_\epsilon + F.\nabla_v
f_\epsilon))(x,v)\Psi(v)dv.
\end{split}
\end{equation*}

Let $\eta>0$. We take $\displaystyle{\lambda=\sup_\epsilon \frac{\Vert
(v.\nabla_x f_\epsilon + F.\nabla_v
f_\epsilon)\Vert_{L^1_{x,v}}\Vert \Psi\Vert_{L^\infty}}{\eta}}$.

 Then we have by Proposition \ref{lemV} :
\begin{equation*}
\begin{split}
\left\Vert \int (R_\lambda (v.\nabla_x f_\epsilon+ F.\nabla_v
f_\epsilon))(x,v)\Psi(v)dv \right\Vert_{L^1_x}\leq& \Vert R_\lambda
(v.\nabla_x f_\epsilon + F.\nabla_v f_\epsilon)
\Vert_{L^1_{x,v}}\Vert \Psi \Vert_{L^\infty_v} \\
\leq& \frac{1}{\lambda}\Vert (v.\nabla_x f_\epsilon + F.\nabla_v f_\epsilon)\Vert_{L^1_{x,v}}\Vert \Psi\Vert_{L^\infty_v} \\
\leq& \eta.
\end{split}
\end{equation*}

Moreover, since $(f_\epsilon)$ is bounded in $L^1(dx \otimes dv)$ and equiintegrable in $x$ and $v$, Proposition \ref{equiXV} implies that the family $(\int
R_\lambda(f_\epsilon)\Psi(v)dv)$ is relatively compact in
$L^1_{x}(B(0,R))$. Finally we can argue as for the end of the proof of Proposition \ref{equiXV}: for all $\eta >0$, there exists $K_\eta \subset
L^1_{x}(B(0,R))$ compact, such that $(\rho_\epsilon)\subset K_\eta +
B(0,\eta)$. So this family is precompact and consequently it is compact since $L^1(B(0,R))$ is a Banach space.

\end{proof}

\section{Mixing properties of the operator $v.\nabla_x  + F.
\nabla_v$}
\label{mix}

\subsection{Free transport case}

In the case when $F=0$, Bardos and Degond in \cite{BD} proved a mixing result (also referred to as a dispersion result for large time asymptotics) which is a key argument in the proof of Theorem \ref{L1} (with $F=0$) by Golse and Saint-Raymond \cite{GolSR}. This kind of estimate was introduced for the study of classical solutions of the Vlasov-Poisson equation in three dimensions and for small initial data.

\begin{lemma}\label{freemixing} Let $f$ be the solution to:
   \begin{equation}
\label{transport} \left\{
    \begin{array}{ll}
      \partial_t f + v.\nabla_x f =0, \\
      f(0,.,.)=f^0.
          \end{array}
  \right.
\end{equation}

Then for all $t>0$:
\begin{equation}
\label{esti}
 \Vert{f(t)}\Vert_{L^\infty_x(L^1_v)} \leq
\frac{1}{|t|^{d}} \Vert{f^0}\Vert_{L^1_x(L^\infty_v)}.
\end{equation}
\end{lemma}

For further results and related questions (Strichartz estimates...), we refer to Castella and Perthame \cite{CasPer} and Salort \cite{S1}, \cite{S2}, \cite{S3}.

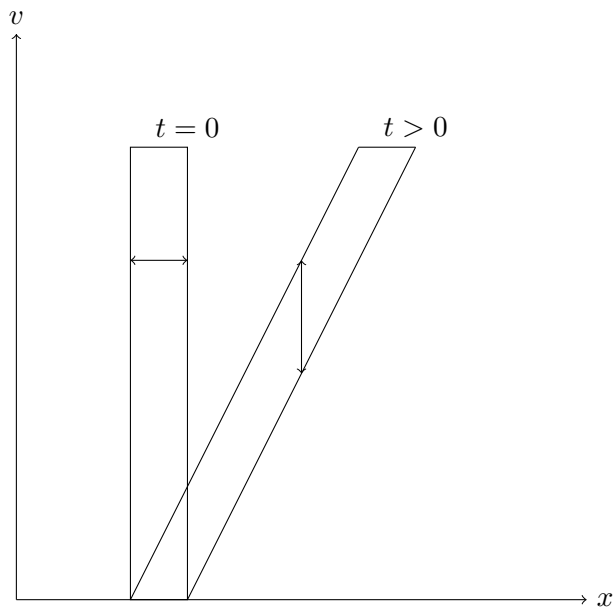
\begin{figure}[htbp]
\begin{center}
\begin{tikzpicture}[scale=1.5]
\draw[->] (0,0) -- (5,0) node[right] {$x$} coordinate(x axis);
\draw[->] (0,0) -- (0,5) node[above] {$v$} coordinate(v axis);
\draw (1,0) rectangle (1.5,4) node[above] {$t=0$} ;
\draw[<->] (1,3) -- (1.5,3);
\draw (1,0) -- (3,4);
\draw (1.5,0) -- (3.5,4);
\draw (3,4) -- (3.5,4) node[above] {$t>0$};
\draw[<->] (2.5,2) -- (2.5,3);
\end{tikzpicture}
\caption{Mixing property for free transport}
\label{melange}
\end{center}
\end{figure}

When $f^0$ is the indicator function of a set with "small" measure with respect to $x$,
then the previous estimate (\ref{esti}) asserts that for $t>0$, $f(t)$ is for any fixed $x$ the indicator function
of a set with a "small" measure in $v$ (at least that we may estimate): this property is crucial for the following. 
In (\ref{esti}), there is blow-up when $t \rightarrow 0$, which is intuitive, but this does not matter
since we have nevertheless a control of the left-hand side for any positive time.

Actually, for our purpose, parameter $t$ is an artificial time (it does not have the usual physical meaning). It appears as an interpolation parameter in Lemma \ref{parts}, and can be taken rather small. This is the reason why local in time mixing is sufficient. We will consequently look for local in time mixing properties. Anyway, the explicit study in Example 2 below shows that the dispersion inequality is in general false for large times when $F\neq 0$.

\begin{proof}[Proof of Lemma \ref{freemixing}] The proof of this result is based on the explicit solution to (\ref{transport}), which is:
\begin{equation*} f(t,x,v)=f^0(x-tv,v)\end{equation*}
We now evaluate:
\begin{eqnarray*}
\Vert f(t) \Vert_{L^\infty_x(L^1_v)} &=& \sup_x \int f^0(x-tv,v)dv
\\&=& \sup_x \int f^0(z,\frac{x-z}{t}) |t|^{-d}dz
\\&\leq& |t|^{-d}\int \Vert f^0(z,.) \Vert_{\infty} dz
\\&\leq& |t|^{-d}\Vert f^0 \Vert_{L^1_x(L^\infty_v)}.
\end{eqnarray*} 
The key argument is the change of variables
$x-tv \mapsto z$, the jacobian of which is equal to $t^{-d}$. 
\end{proof}

We intend to do the same in the more complicated case when $f$ is the solution of a transport equation with
$F \neq 0$. Let us mention that in \cite{BD}, Bardos and Degond actually prove the dispersion result
for non zero force fields but with a polynomial decay in time. Here, this is not the case (the field $F$
does not even depend on time $t$), but we will prove that the result holds anyway for small times.

\subsection{Study of two examples}

In the following examples, $f$ is the explicit solution to the transport equation (\ref{transport1}) with an initial condition $f^0$ 
and a force deriving from a potential. 
\\

\textbf{Example 1} 

Force $F= -\nabla_x V$, with $V = -|x|^2/2$.

 Let $f$ be the solution to:
   \begin{equation}
 \left\{
    \begin{array}{ll}
      \partial_t f + v.\nabla_x f + x.\nabla_v f =0, \\
      f(0,.,.)=f^0.
          \end{array}
  \right.
\end{equation}
The effect of such a potential will be to make the particles escape faster to infinity.  So we expect to have results very similar to those of lemma \ref{freemixing}. 

After straightforward computations we get :
$$f(t,x,v)=f^0\left(x \left(\frac{e^t+e^{-t}}{2}\right)+v  
\left(\frac{e^{-t}-e^{t}}{2} \right),x \left(\frac{e^{-t}-e^{t}}{2}  
\right)+v \left(\frac{e^t+e^{-t}}{2} \right) \right),$$

which allows to show the same dispersion estimate with a factor $\frac{e^t-e^{-t}}{2}$ instead of $t$.

For all $t>0$, we have:
\begin{equation} \Vert{f(t)}\Vert_{L^\infty_x(L^1_v)} \leq
\frac{2^d}{(e^t-e^{-t})^{d}} \Vert{f^0}\Vert_{L^1_x(L^\infty_v)}.
\end{equation}

\textbf{Example 2}

(Harmonic potential) Force $F= -\nabla_x V$, with $V =|x|^2/2$.

Let $f$ be the solution to:
   \begin{equation}
 \left\{
    \begin{array}{ll}
      \partial_t f + v.\nabla_x f - x.\nabla_v f =0, \\
      f(0,.,.)=f^0.
          \end{array}
  \right.
\end{equation}

With such a potential, particles are expected to be confined and consequently do not drift to infinity.
For this reason, it is hopeless to prove the analogue of Lemma \ref{freemixing} for large times (here there is no dispersion). As mentioned before, it does not matter since
we only look for a result valid for small times. We expect that there is enough mixing in the phase space to prove the result.

After straightforward computations we explicitly have :
$$f(t,x,v)=f^0(x\cos t - v\sin t,v\cos t + x \sin t).$$ 
We observe here that the solution $f$ is periodic with respect to time. Thus, as expected, it is not possible to prove any decay when $t\rightarrow +\infty$; nevertheless we can prove a mixing  estimate with a factor   $\vert \sin t \vert$ instead of $t$.

For all $t>0$:
\begin{equation} \Vert{f(t)}\Vert_{L^\infty_x(L^1_v)} \leq
\frac{1}{\vert \sin t\vert^{d}} \Vert{f^0}\Vert_{L^1_x(L^\infty_v)}.
\end{equation}
Of course, this estimate is useless when $t= k \pi, k \in \mathbb{N}^*$.

\begin{remark}
We notice that $\frac{e^t-e^{-t}}{2} \sim_0 t$ and $\sin (t)\sim_0 t$, which seems encouraging.
\end{remark}

\subsection{General case : $F$ with Lipschitz regularity}

The study of these two examples suggests that at least for small times, the mixing estimate is still satisfied, maybe
with a corrector term which does not really matter. 

One nice heuristic way to understand this is to see that since $F$ is quite smooth, the dynamics associated to the operator $v.\nabla_x+ F.\nabla_v$ is expected to be close to those of free transport, at least for small times.

Let $X(t;x,v)$ and $V(t;x,v)$ be the diffeomorphisms introduced in the method of characteristics in Section \ref{first} and defined in (\ref{carac}).

Using Taylor's formula, we get by definition of $X$  
and $V$ :
\begin{eqnarray*}
X(-t;x,v)=x-tv+  \int_0^t (t-s) F( X(-s;x,v))ds.
\end{eqnarray*}

We recall Rademacher's theorem which asserts that $W^{1,\infty}$ functions are almost everywhere derivable. Hence, using Lebesgue domination theorem,
we get:

\begin{eqnarray}
\label{id}
\partial_v [X(-t;x,v)]=-tId +\int_0^t (t-s) \nabla_x F(X(-s;x,v)) 
\partial_v [X(-s;x,v)]ds . 
\end{eqnarray}

We deduce the estimate :
\begin{eqnarray*}
\Vert \partial_v [X(-t;x,v)]\Vert_\infty \leq t +  \int_0^t (t-s)
\Vert \nabla_x F \Vert_\infty \Vert \partial_v
[X(-s;x,v))]\Vert_\infty ds.
\end{eqnarray*}

Gronwall's lemma implies then that:

\begin{equation}
\label{Gronwall} \Vert \partial_v [X(-t;x,v)]\Vert_\infty
\leq t e^{\frac{t^2}{2}\Vert \nabla_x F \Vert_\infty}.
\end{equation}

We can also take the determinant of identity (\ref{id}) :

\begin{equation}
\label{determ}
\begin{split}
\det (\partial_v [X(-t;x,v)] )=& \\(-t)^d& \det \left(Id -   \frac{1}{t}\int_0^t (t-s) \nabla_x F(X(-s;x,v)) 
\partial_v [X(-s;x,v)]ds \right).
\end{split}
\end{equation}

The right-hand side is the determinant of a matrix of the form 
$Id+A(t)$ where $A$ is a matrix whose $L^\infty$ norm is small for small times  $t$ (one can use
estimate (\ref{Gronwall}) to ensure that $\Vert
A(t) \Vert_\infty=o(t)$). Consequently, in a neighborhood of $0$,
for any fixed $x$, $\partial_v [X(-t;x,v)]$ is invertible. Furthermore the map $v\mapsto X(-t;x,v))$  is
injective for small positive times. Indeed, let $v \neq v'$. We compare:

\[
X(-t;x,v')-X(-t;x,v)=t(v-v')+ \int_0^t (t-s) [F( X(-s;x,v'))-F( X(-s;x,v)) ]ds.
\]
Consequently we have:
\[
\vert X(-t;x,v')-X(-t;x,v)\vert \leq t\vert v-v' \vert + \int_0^t (t-s) \Vert \nabla_x F \Vert_{L^\infty}  \vert X(-s;x,v')- X(-s;x,v) \vert ds.
\]
Thus, by Gronwall inequality we obtain:
\[
\vert X(-t;x,v')-X(-t;x,v)\vert  \leq t\vert v-v' \vert e^{\frac{t^2}{2}\Vert \nabla_x F\Vert_{L^\infty}}.
\]
Finally we observe that:
\begin{equation*}
\begin{split}
\vert X(-t;x,v')-X(-t;x,v)\vert \geq & t\vert v-v' \vert - \left\vert\int_0^t (t-s) [F( X(-s;x,v'))-F( X(-s;x,v)) ]ds \right\vert\\
\geq & t\vert v-v' \vert - \int_0^t (t-s) \Vert \nabla_x F \Vert_{L^\infty}  \vert X(-s;x,v')- X(-s;x,v) \vert ds \\
\geq & \vert v-v' \vert \left(t - \int_0^t (t-s) s e^{\frac{s^2}{2}\Vert \nabla_x F\Vert_{L^\infty}} \Vert \nabla_x F \Vert_{L^\infty}ds\right).
\end{split}
\end{equation*}
Consequently, there is a maximal time $\tau_0>0$, depending only on $\Vert \nabla_x F \Vert_{L^\infty}$ such that for any $\vert t \vert\leq \tau_0$, we have :
\[
\vert X(-t;x,v')-X(-t;x,v)\vert \geq \frac t 2 \vert v-v' \vert.
\]
This proves our claim.

Thus, by the local inversion theorem, this map is a $\mathcal{C}^1$
diffeomorphism on its image. 

We have now the following elementary quantitative estimate :

\begin{lemma}
Let $t\mapsto A(t)$ be a continuous map defined on a neighborhood of $0$, such that $\Vert A(t) \Vert_\infty=o(t)$. Then for small times:
\begin{equation*}
\det(Id+A(t))\geq 1 - d! \Vert A(t) \Vert_\infty.
\end{equation*}
We recall that $d$ is the space dimension and $d!=1\times 2 \times  ... \times d$.
\end{lemma}

We apply this lemma to (\ref{determ}), which allows us to say that there exists a maximal time $\tau>0$ such that for any $\vert t \vert\leq \tau$, we have :

\begin{equation}
\label{det}
|\det (\partial_v [X(-t;x,v)]|^{-1} \leq 2\vert t\vert^{-d}.
\end{equation}

We have proved that $v \mapsto X(-t;x,v)$ is a $\mathcal{C}^1$
diffeomorphism such that the jacobian of its inverse satisfies (\ref{det}) in a neighborhood of  $t=0$. We can consequently conclude as in the proof of Lemma \ref{freemixing} (by performing the change of variables
$X(-t;x,v)\mapsto v$).

As a result we have proved the proposition :

\begin{proposition}
\label{mixing}
 Let $ F(x)$ be a Lipschitz vector field. There exists a maximal time $\tau>0$  (depending only on $\Vert \nabla_x F \Vert _{L^{\infty}}$)  such that, if $f$ is the solution to
the transport equation:

   \begin{equation}
 \left\{
    \begin{array}{ll}
      \partial_t f + v.\nabla_x f +F.\nabla_v f=0, \\
      f(0,.,.)=f^0\in L^p(dx \otimes dv).
    \end{array}
  \right.
\end{equation}
Then:

\begin{equation}
\forall \vert t \vert \leq \tau, \Vert{f(t)}\Vert_{L^\infty_x(L^1_v)} \leq
\frac{2}{|t|^{d}} \Vert{f^0}\Vert_{L^1_x(L^\infty_v)}.
\end{equation}
\end{proposition}

\begin{remark}
If one writes down more explicit estimates, it can be easily shown that $\tau$ is bounded from below by $T$ defined as the only positive solution to the equation:

\begin{equation}
 \frac{d!}{3 }
 \Vert \nabla_x F\Vert_{\infty}T^2 e^{\Vert \nabla_x F\Vert_{\infty} \frac{T^2}{2}}  =1.
 \end{equation}
\end{remark}

\begin{remark}
Of course, one can replace the factor $2$ in the mixing estimate by any $q>1$ (and the maximal time $\tau$ will depend also on $q$).
\end{remark}

\section{From local equiintegrability in velocity to local equiintegrability in position and velocity}

In this section, we finally proceed as in \cite{GolSR}, with some slight modifications adapted to our case. We start from the following Green's formula :
\begin{lemma}
\label{parts}
Let $f \in L^1(dx\otimes dv)$ with compact support such that $
v.\nabla_x f + F. \nabla_v f \in L^1(dx\otimes dv)$.
Then for all $\Phi^0 \in L^\infty(dx\otimes dv)$, we have for all $t
\in \mathbb{R}^*_+$ :
\begin{equation}
\begin{split}
\int\ f(x,v) \Phi^0(x,v)dxdv=&\int\
f(x,v)\Phi(t,x,v)dxdv\\-&\int_0^t\int \Phi(s,x,v)(v.\nabla_x f +
F. \nabla_v f)dsdxdv,
\end{split}
\end{equation}
where $\Phi$ is the solution to:
\begin{equation}
\label{trans}
\left\{
    \begin{array}{ll}
      \partial_t \Phi+ v.\nabla_x \Phi + F.\nabla_v \Phi= 0 \\
      \Phi_{\vert t=0}=\Phi_0.
    \end{array}
  \right.
\end{equation}
\end{lemma}

\begin{proof}
We have for all  $t>0$, $\int_\Omega
f(x,v)(\partial_t+v.\nabla_x+F.\nabla_v)\Phi(s,x,v)dsdxdv=0$, where $\Omega=]0,t[\times
\mathbb{R}^d\times \mathbb{R}^d$. We first have:
\begin{equation*}
\int_\Omega f(x,v)\partial_t \Phi(s,x,v)dsdxdv=\int\int f(x,v)
\Phi(t,x,v)dxdv - \int\int f(x,v) \Phi^0(x,v)dxdv.
\end{equation*}
Finally, by Green's formula we obtain:
\begin{equation*}
\begin{split}
\int_0^t \int f(x,v)(v.\nabla_x +& F.\nabla_v)\Phi(s,x,v)dsdxdv  \\= -&\int_0^t \int
\Phi(s,x,v)(v.\nabla_x+F.\nabla_v)f(x,v)dsdxdv.
\end{split}
\end{equation*}
There is no contribution from the boundaries since $f$ is compactly supported. \end{proof}

\begin{lemma}
\label{lem1}
Let $(f_\epsilon)$ a bounded family of $L^1_{loc}(dx\otimes dv)$
locally integrable in $v$ such that $(v.\nabla_x f_\epsilon +
F.\nabla_v f_\epsilon)$ is a bounded family of
$L^1_{loc}(dx\otimes dv)$. Then for all $\Psi \in \mathcal{C}^1_c(\mathbb{R}^d)$,  
such that $\Psi \geq 0$,
the family $\rho_\epsilon(x)=\int \vert f_\epsilon(x,v) \vert \Psi(v)dv 
$ is locally equiintegrable.
\end{lemma}

\begin{proof}
Let $K_1$ be a compact subset of $\mathbb{R}^d$. We want to prove that
$(\mathbbm{1}_{K_1}\rho_\epsilon(x))$ is equiintegrable. Without loss of generality, we can assume as previously that the $f_\epsilon$ 
are supported in the same compact support $K=K_1 \times K_2$  and that $\Psi$ is compactly
supported in $K_2$. Furthermore, the formula $\nabla \vert f_\epsilon  
\vert = \operatorname{sign}(f_\epsilon)\nabla f_\epsilon$  
 shows that the $\vert f_\epsilon \vert$  
satisfy the same assumptions of equiintegrability and 
$L^1$ boundedness as the family $(f_\epsilon)$. 
For the sake of readability, we will thus assume that $f_\epsilon$ are almost everywhere 
non-negative instead  of considering $\vert f_\epsilon \vert 
$. Finally we may assume that $\Vert \Psi \Vert_\infty=1$ (multiplying by a constant does not change the equiintegrability property).

The idea of the proof is to show that thanks to the mixing properties established previously,
the equiintegrability in $v$ provides some  equiintegrability in $x$.
\par

Let $\eta>0$. By definition of the local equiintegrability in $v$, we obtain a parameter
 $\alpha>0$  associated to $K$ and $\eta$. We also consider parameters $\alpha'>0$ and $t\in]0,\tau[$ (where $\tau$ is the maximal time in Proposition \ref{mixing}) to be fixed ultimately. We mention that $t$ will be chosen only after $\alpha'$ is fixed.

Let $A$ a bounded mesurable subset included in $K_1$ with $|A|\leq
\alpha'$. We consider $\Phi^0(x,v)=\mathbbm{1}_A(x)$ and $\Phi$ the solution of the transport
equation (\ref{trans}) with $ 
\Phi^0$ as initial data.

Observe now that we have $\Vert \Phi^0
\Vert_{L^1_x(L^\infty_v)}=|A|$. Moreover, since $\Phi^0$ takes its values in $\{0,1\}$, it is also the case for $\Phi$ (this is a plain
consequence of the transport of the data).
\par
We define for all $s>0$ and for all $x\in \mathbb{R}^d$, the set $A(s)_x=\{v\in\mathbb{R}^d,
\Phi(s,x,v)=1\}$. At this point of the proof, we make a crucial use of the mixing property stated in Proposition \ref{mixing} :
\begin{eqnarray*}
\sup_x |A(t)_x|&=&\sup_x \int \Phi(t,x,v)dv \\
&=& \Vert \Phi(t,.,.) \Vert_{L^\infty_x(L^1_v)} \\
&\leq& 2|t|^{-d} \underbrace{\Vert \Phi^0
\Vert_{L^1_x(L^\infty_v)}}_{|A|\leq \alpha'} \\
&\leq& \alpha,
\end{eqnarray*}
if we choose $\alpha'$ satisfying $\alpha'< \frac{1}{2} t^D \alpha$.

Thanks to Lemma \ref{parts}:
\begin{equation*}
\begin{split}
\int f(x,v) \Psi(v)\Phi^0(x,v)dxdv=&\int
f(x,v)\Psi(v)\Phi(t,x,v)dxdv\\-&\int_0^t\int
\Phi(s,x,v)(v.\nabla_x  + F. \nabla_v
)(f_\epsilon(x,v)\Psi(v))dxdvds.
\end{split}
\end{equation*}

In other words, the operator $\partial_t + v.\nabla_x + F.\nabla_v$ has  transported   
the indicator function and has transformed a subset small in $x$ into a subset small 
in $v$.

By definition of $\rho_\epsilon$, we have:
$$
\int f(x,v) \Psi(v)\Phi^0(x,v)dxdv = \int \mathbbm{1}_A(x) \rho_\epsilon(x)dx.
$$
By definition of $A(t)_x$ we also have:
$$
\int f(x,v)\Psi(v)\Phi(t,x,v)dxdv=\int\left(\int_{A(t)_x}f_\epsilon(x,v)\Psi(v)dv \right)dx.
$$

Thus,  since $(f_\epsilon)$ are locally equiintegrable in $v$ we may evaluate:
\begin{equation*}
\begin{split}
\int\left(\int_{A(t)_x}f_\epsilon(x,v)\Psi(v)dv \right)dx \leq&  \int \int_{A(t)_x} |f_\epsilon| \mathbbm{1}_{K}
\underbrace{\Vert \Psi\Vert_{\infty}}_{=1}dxdv \\
\leq &\eta .
\end{split}
\end{equation*}

Finally we have:
\begin{equation*}
\begin{split}
\int \mathbbm{1}_A(x) \rho_\epsilon(x)dx=&
\int\left(\int_{A(t)_x}f_\epsilon(x,v)\Psi(v)dv \right)dx \\ -& 
\int_0^t\int
\Phi(s,x,v)(v.\nabla_x + F. \nabla_v
)(f_\epsilon(x,v)\Psi(v) )dsdxdv
\end{split}
\end{equation*}
\begin{equation*}
\begin{split}
\leq&\eta + \int_0^t\int
|\Phi(s,x,v)||(v.\nabla_x + F. \nabla_v
)(f_\epsilon(x,v)\Psi(v))|dsdxdv \\
\leq& \eta + t\left[\underbrace{\Vert \Psi \Phi\Vert_\infty}_{\leq1}
\Vert v.\nabla_x f_\epsilon + F.\nabla_v f_\epsilon \Vert_1
+ \underbrace{\Vert \Phi\Vert_\infty}_{=1} \Vert F.\nabla_v
\Psi(v)\Vert_\infty \Vert f_\epsilon \Vert_1\right]\\
\leq& 2\eta,
\end{split}
\end{equation*}

by taking $t$ sufficiently small:
$$t<\frac{\eta}{\sup_\epsilon \Vert
v.\nabla_x f_\epsilon + F.\nabla_v f_\epsilon \Vert_1+\Vert
F.\nabla_v \Psi(v)\Vert_\infty \Vert f_\epsilon \Vert_1}.$$

This finally proves that  $(\rho_\epsilon)$ is locally 
equiintegrable in $x$.
\end{proof}

\begin{lemma}
\label{lem2}
Let $(g_\epsilon)$a bounded family of $L^1_{loc}(dx\otimes dv)$
locally equiintegrable in $v$. If for all $\Psi \in
\mathcal{C}^1_c(\mathbb{R}^d)$ such that $\Psi \geq 0$, $x \mapsto \int |g_ 
\epsilon(x,v)|\Psi(v)dv$ is
locally equiintegrable (in $x$), then $(g_\epsilon)$ is
locally equiintegrable in $x$ and $v$.
\end{lemma}

\begin{proof}
Let $K$ be a compact subset of $\mathbb{R}^d \times \mathbb{R}^d$. We
want to prove that $(\mathbbm{1}_K g_\epsilon)$ is
equiintegrable in $x$ and $v$. As before, we can clearly assume that the  $g_\epsilon$
are compactly supported in $K$.
\par
Let $\eta>0$. By definition of the local equiintegrability in $v$
for $(g_\epsilon)$, we obtain $\alpha_1>0$ associated to $\eta$ and
$K$.

\par
Let $\Psi \in \mathcal{C}^1_c(\mathbb{R}^d)$ a smooth non-negative and compactly supported function such that  $\Psi\equiv 1$ on $p_v(K)$ (where $p_v(K)$ is the projection of $K$ on $\mathbb{R}^d_v$). 
By assumption, there exists $\alpha_2
 >0$ such that for any $A \subset \mathbb{R}^d$ measurable set satisfying
$|A|\leq \alpha_2$,
\begin{equation*}
\int_A\left(\int |g_\epsilon|\Psi dv\right)dx<\eta.
\end{equation*}
Let $B$ a measurable subset of $\mathbb{R}^d \times \mathbb{R}^d$
such that $|B|<\inf (\alpha_1^2,\alpha_2^2)$. We define for all $x \in
\mathbb{R}^d$, $B_x=\{v \in \mathbb{R}^d, (x,v) \in B \}$.
\par
We consider now $E=\{x \in \mathbb{R}^d, |B_x|\leq |B|^{1/2}\}$ : this is the
subset of $x$ for which there exist  few $v$ such that  
$(x,v)\in B$. Consequently for this subset, we can use the
local equiintegrability in $v$.

Concerning $B \backslash E$, on the contrary, we can not use this property,  but 
thanks to Chebychev's inequality we show that this subset is 
of small measure, which allows us to use this time the
 local equiintegrability in $x$ of $\int |g_ \epsilon(x,v)|\Psi(v)dv$ :
\begin{eqnarray*}
|E^c|&=&\vert \{x \in \mathbb{R}^d, |B_x|> |B|^{1/2}\} \vert \\
&\leq& \frac{|B|}{|B|^{1/2}} \\
&\leq& \alpha_2.
\end{eqnarray*}

Hence we have :
\begin{eqnarray*}
\int \mathbbm{1}_B |g_\epsilon|dxdv &\leq& \int_E \left(\int_{B_x}
|g_\epsilon|dv\right)dx+\int_{E^c} \left(\int |g_\epsilon|dv \right)dx  
\\
&\leq& \eta  + \int_{E^c} \left(\int
|g_\epsilon|\Psi(v)dv \right)dx \\
&\leq& 2\eta.
\end{eqnarray*}

This shows the expected result.
\end{proof}

We are now able to conclude the proof of Theorem \ref{L1}.

\begin{proof}[End of the proof of Theorem \ref{L1}]
 If we successively apply Lemmas \ref{lem1} and \ref{lem2}, we deduce that  the family $(f_\epsilon)$ is locally
   equiintegrable in $x$ and $v$.
   
   Finally we have shown in Section \ref{first} that the first point implies the second.

\end{proof}

\begin{acknowledgements}
The author would like to thank his advisor, Laure Saint-Raymond, for suggesting the subject and for her help. He is also indebted to the referee for several valuable remarks and suggestions which helped to improve the presentation of the paper.
\end{acknowledgements}


\begin{thebibliography}{30}




\bibitem{Ago} %(MR0753365)
     \newblock V.I. {Agoshkov},
     \newblock \emph{Spaces of functions with differential-difference characteristics and the smoothness of solutions of the transport equation},
     \newblock Dokl. Akad. Nauk SSSR, \textbf{276} (1984), 1289-1293.

\bibitem{BD} %(MR0794002)
     \newblock Claude {Bardos} and Pierre {Degond},
     \newblock \emph{Global existence for the Vlasov-Poisson equation in $3$ space variables with small data},
     \newblock Ann. Inst. H. Poincar\'e Anal. Non Lin\'eaire, \textbf{2} (1985), 101-118.


\bibitem{BZ} %(MR1259108)
     \newblock Max {Bézard},
     \newblock \emph{Régularit\'e $L^p$ précis\'ee des  
moyennes
dans les équations de transport},
     \newblock Bull. Soc. math. France, \textbf{122} (1994), 29-76.


\bibitem{CasPer} %(MR2082924)
     \newblock François {Castella} and Benoit
{Perthame},
     \newblock \emph{Estimations de Strichartz pour les \'equations de
transport cin\'etique},
     \newblock  C. R. Acad. Sci. Paris S\'er. I Math, \textbf{322} (1996), 535-540.


\bibitem{DPLbol} %(MR1014927)
     \newblock Ronald J. {DiPerna} and Pierre-Louis
{Lions},
     \newblock \emph{On the Cauchy problem for Boltzmann equations: global existence and weak stability},
     \newblock Ann. of Math., \textbf{130} (1989), 321-366.


\bibitem{DPLvm} %(MR1003433)
     \newblock Ronald J. {DiPerna} and Pierre-Louis
{Lions},
     \newblock \emph{Global weak solutions of Vlasov-Maxwell systems},
     \newblock  Comm. Pure Appl. Math, \textbf{42} (1989), 729-757.


\bibitem{DPLM} %(MR1127927)
     \newblock Ronald J. {DiPerna} and Pierre-Louis {Lions},
Yves {Meyer},
     \newblock \emph{$L^p$ regularity of velocity averages},
     \newblock  Ann. Inst. H. Poincar\'e Anal. Non Lin\'eaire, \textbf{8} (1991), 271-287.


\bibitem{GPS} %(MR0808622)
     \newblock François {Golse}, Benoit {Perthame} and Rémi  
{Sentis},
     \newblock \emph{Un r\'esultat de compacit\'e pour les
\'equations de transport et application au calcul de la limite de la  
valeur
propre principale d'un op\'erateur de transport},
     \newblock C. R. Acad. Sci. Paris S\'er. I Math, \textbf{301} (1985), 341-344.


\bibitem{GLPS} %(MR0923047)
     \newblock François {Golse}, Pierre-Louis {Lions},  
Benoit {Perthame} and Rémi
{Sentis},
     \newblock \emph{Regularity of the Moments of the Solution of a
Transport Equation},
     \newblock  J. Funct. Anal., \textbf{76} (1988), 110-125.



\bibitem{GolSR} %(MR1903763)
     \newblock François {Golse} and Laure
{Saint-Raymond},
     \newblock \emph{Velocity averaging in $L^1$ for the
transport equation},
     \newblock C. R. Acad. Sci. Paris S\'er. I Math, \textbf{334} (2002), 557-562.


\bibitem{GSR} %(MR2025302)
     \newblock François {Golse} and Laure
{Saint-Raymond},
     \newblock \emph{The Navier-Stokes limit of the Boltzmann equation for bounded collision kernels},
     \newblock Invent. Math., \textbf{155} (2004), 81-161.




\bibitem{J} %(MR2597793)
     \newblock Pierre-Emmanuel {Jabin},
     \newblock \emph{Averaging lemmas and dispersion estimates for kinetic equations},
     \newblock Riv. Mat. Univ. Parma, \textbf{8} (2009), 71-138.


\bibitem{JV} %(MR2096303)
     \newblock Pierre-Emmanuel {Jabin} and Luis {Vega},
     \newblock \emph{A  
Real Space Method for Averaging Lemmas},
     \newblock J. de
Math. Pures et Appl., \textbf{83} (2004), 1309-1351.



\bibitem{PS} %(MR1634024)
     \newblock Benoit {Perthame} and Panagiotis {Souganidis},
     \newblock \emph{A limiting case for velocity averaging},
     \newblock Ann. Sci. Ecole Norm. Sup. , \textbf{31} (1998), 591--598.

\bibitem{S1} %(MR2224407)
     \newblock Delphine {Salort},
     \newblock \emph{Weighted dispersion and Strichartz estimates for the Liouville equation in $1D$},
     \newblock Asymptot. Anal., \textbf{47} (2006), 85-94.



\bibitem{S2} %(MR2292518)
     \newblock Delphine {Salort},
     \newblock \emph{Dispersion and Strichartz estimates for the Liouville equation},
     \newblock J. Differential Equations, \textbf{233} (2007), 543-584.


\bibitem{S3} %(MR2498433)
     \newblock Delphine {Salort},
     \newblock \emph{Transport equations with unbounded force fields and application to the Vlasov-Poisson equation},
     \newblock Math. Models Methods Appl. Sci., \textbf{19} (2009), 199-228..


\end{thebibliography}
\end{document}